\documentclass[12pt]{amsart}

\usepackage[T1]{fontenc}  % Proper font encoding
\usepackage[utf8]{inputenc}  % UTF-8 encoding (only needed for legacy TeX engines)
\usepackage{lmodern}  % Better scalable fonts
\usepackage{amsmath, amssymb, amsfonts, amscd}  % Standard AMS packages

\usepackage{graphicx}  % Image support
\usepackage{xcolor}  % Color support
\usepackage{subcaption}  % Modern subfigures (avoid deprecated subfigure package)
\usepackage[all]{xy}  % Xy-pic for diagrams

\usepackage{geometry}
\newtheorem{theorem}{Theorem}[section]
\newtheorem{lemma}[theorem]{Lemma}
\theoremstyle{definition}

\newtheorem{corollary}[theorem]{Corollary}

\theoremstyle{remark}

\usepackage{hyperref}  % Must be loaded last
\hypersetup{
    colorlinks=true,
    linkcolor=blue,
    citecolor=blue,
    urlcolor=blue,
    pdfauthor={Your Name},
    pdftitle={Your Paper Title}
}

\newcommand{\R}{{\mathbb{R}}}

\begin{document}
\title{gordian unlinks}

%\date{\today}

\author{Jos\'{e} Ayala}
\address{Universidad de Tarapac\'a, Iquique, Chile}
\email{jayalhoff@gmail.com}
\author{Joel Hass}
\address{Department of Mathematics, University of California, Davis California 95616}
\email{jhass@ucdavis.edu}
\thanks{Research partially supported by Fondecyt Grant \#11220579}

%\baselineskip=20 true pt

\begin{abstract} \baselineskip=1.2\normalbaselineskip 
This paper gives the first examples of gordian unlinks. The components of these unlinks cannot be separated while maintaining constant length and thickness. We construct infinite families of 2-component gordian unlinks and also construct $n$-component gordian unlinks for each $n \geq 2$.
\end{abstract}

\maketitle \baselineskip=1.3\normalbaselineskip 

\section{Introduction}

A knot or link is {\em thick} if it has an embedded neighborhood of some fixed radius. After scaling, the radius can be assumed to be one. A {\em gordian unlink} is a thick unlink that cannot be split by an isotopy that preserves length and thickness. This means that no such isotopy moves the link so that any two components are separated by a plane. In this paper, we give the first examples of gordian unlinks. Two examples are shown in Figure \ref{fig:example1}.

\begin{figure}[htbp]
   \centering
   \includegraphics[width=.4\linewidth]{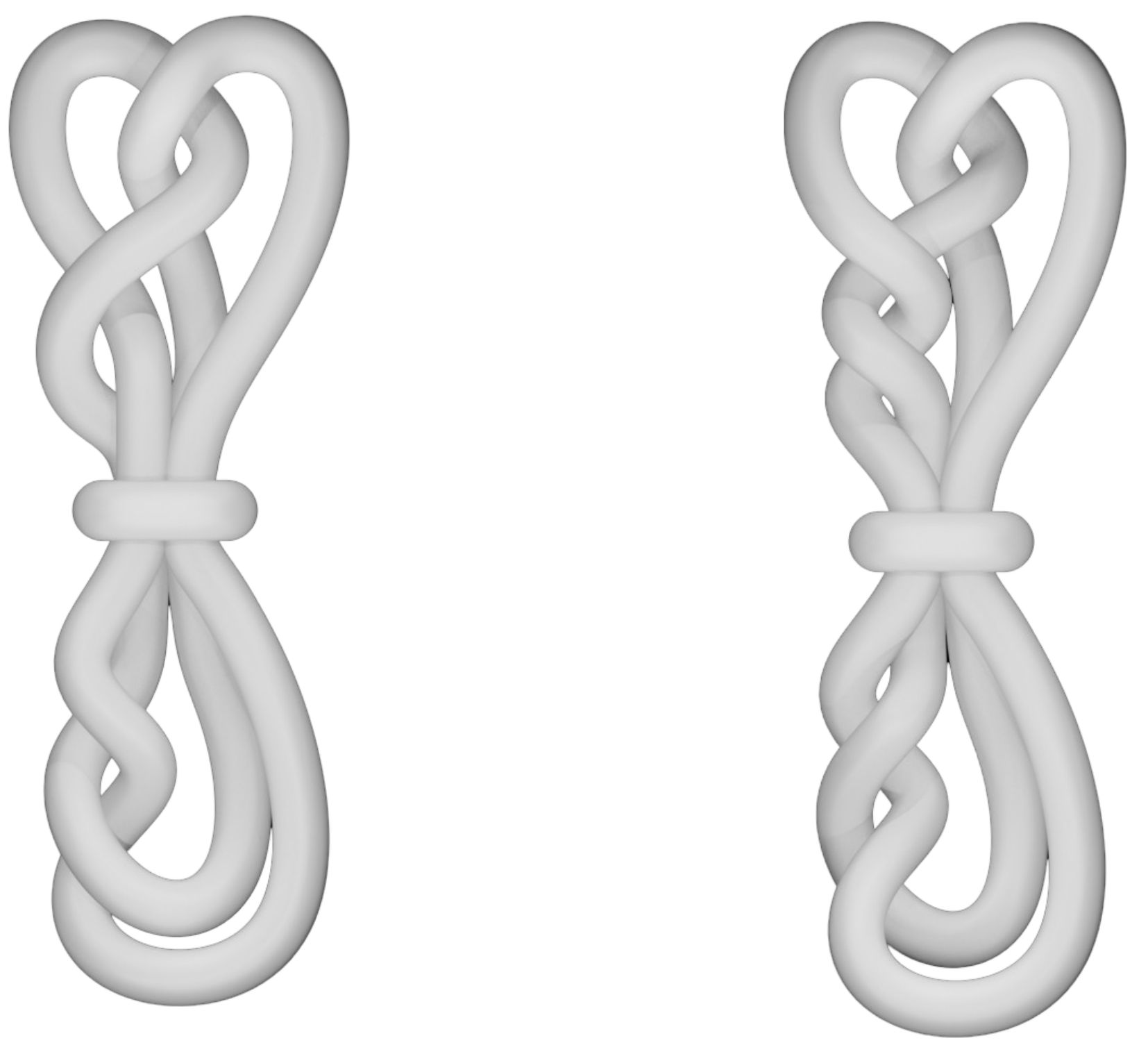}
    \caption{Two of the simplest examples in a family $L(m,n)$ of gordian unlinks. The unlink on the left is  $L(1,1)$, and on the right is $L(1, -1)$.}
    \label{fig:example1}
   \end{figure}

Simple closed curves in $\R^3$ that model physical curves, such as ropes or proteins, have a positive thickness. This is in contrast to classical knot theory, which studies 1-dimensional curves without thickness. Knot theory is playing an increasing role in the study of biological and physical processes \cite{{flapanwong, seguraetal, tubiana}}, and the properties of thick knots are often more relevant in these applications than the properties of classical knots.

We use the standard definition of thickness, in which a curve is thick if its reach is greater than or equal to one \cite{gm, CKS}. This means that any point with distance less than one from the link has a unique closest point on the link. Two thick links are equivalent if there is an isotopy between them that maintains this thickness condition, a {\em thick isotopy}. 
Isotopy classes of thick links differ from those of classical links, but this difference is surprisingly difficult to prove. The first examples were given by Coward and Hass \cite{cowardhass}, with a construction of 2-component thick links that are topologically split but cannot be split by a thick isotopy. These examples rely on the knottedness of one component. Kusner and Kusner later constructed a link with seven unknotted components, but with some pairs of components having a non-zero linking number. They showed that these links have two thick configurations that are not isotopic through a thick isotopy \cite{kusnerkusner}. Their method requires having non-trivial linking numbers between distinct components, and again does not extend to an unlink. Ayala considered 2-component unlinks that are thick and also satisfy an additional restriction that bounds the curvature of the surfaces forming the boundary of the tubular neighborhood of the link \cite{thingor1}. He showed that these unlinks are not isotopic to the standard 2-component unlink through thick isotopies that also maintain this additional condition, but his methods do not apply without these additional restrictions. 

In this article, we construct the first examples of 2-component gordian unlinks. These examples have one component which is a convex curve on the $xy$ plane of length $8+4\pi$. The second component, also unknotted, weaves back and forth through the first.
Figure~\ref{fig:example1} shows two examples of such links. The method also works to construct 
$n$-component gordian unlinks for $n>2$.
 
The existence of these links has consequences for the configuration space of links. Hatcher's proof of the Smale conjecture implies that the space of smooth unknotted loops in $\mathbb{R}^3$ is homotopy equivalent to the space of great circles in $S^2$ \cite{hatcher1}. Brendle and Hatcher extended this to $n$-component unlinks, showing that the space of smooth links in $\mathbb{R}^3$ that are isotopic to the $k$-component unlink is homotopy equivalent to the space of configurations of $k$ unlinked round circles \cite{hatcher2}. The existence of gordian unlinks shows that this result is false for thick links. The space of $k$-component thick unlinks is not homotopy equivalent to the space of configurations of $k$ unlinked round circles, for any $k \ge 2$. 

We note that R. Kusner and W. Kusner have announced a construction of different gordian unlinks \cite{kusnerkusner2}.

A {\em gordian unknot}, the case where $k=1$, is an unknotted loop with fixed length and thickness that cannot be deformed into a round circle by an isotopy preserving length and thickness. The existence of a gordian unknot remains open.

\noindent
{\bf Acknowledgments.}
We thank Rafael Gonzalez for rendering the 3D models of the thick knots shown in the figures.

\section{Preliminaries}

We review some properties of coned surfaces. 
Let $\gamma:[0,1]\to \R^3$ be a closed curve and $P \in \R^3$ be a point. 
The surface $C(\gamma,P)$ defined by $P + t\gamma(s)$ with $0 \le t < \infty, ~ 0 \le s \le 1$ is called the cone {\em over $\gamma$ with the cone point $P$.} We always take the cone point to be located at the center of mass, or centroid, of $\gamma$, and denote the cone by $C(\gamma)$.
The unit sphere around $P$ intersects $C(\gamma)$ in a curve $\bar \gamma \subset S^2$. 
The {\em cone  angle $\theta$}  of $C(\gamma)$ is defined to be the length of $\bar \gamma$ and $D(\gamma)$ is the disk that $\gamma$ bounds in $C(\gamma)$.
 
The curvature of the curve $\gamma$ and the surface $C(\gamma)$ are generally defined using second derivatives, which may not exist for a $C^{1,1}$ curve.  
When $\gamma$ is piecewise $C^2$ then $C(\gamma)$ inherits a flat metric away from the cone point.
When the cone angle satisfies $\theta \ge 2\pi$ then $C(\gamma)$ is a 2-dimensional CAT(0) space.
In these spaces, there is a unique geodesic connecting any two points, and the notions of half-space and convex hull are well-defined. 
Since thick curves are not $C^2$ in general, we take care to extend our results beyond the setting of smooth curves. 
For many results, it suffices to assume that $\gamma$ is rectifiable, so that lengths are defined.
Cantarella, Kusner, and Sullivan \cite{CKS} established $C^{1,1}$ regularity for thick links and studied configurations of thick links with minimal length, called ideal links \cite{idealbook}.

We now state a number of preliminary lemmas. 

\begin{lemma} \label{isoperimetric}
Given a rectifiable curve $\gamma$, the planar isoperimetric inequality holds for $D(\gamma)$, 
$$
 \mbox{area}(D(\gamma)) \le (\ell(\gamma))^2/4\pi .
$$
\end{lemma}
\begin{proof}
The rectifiable curve $\gamma$ is a pointwise limit of smooth curves $\gamma_i$ with length limiting to $\ell(\gamma)$. The curve $\gamma_i$ bounds a disk $D(\gamma_i)$ in the cone $C(\gamma_i)$ coned to the center of mass $P_i$ of $\gamma_i$. The cones $C(\gamma_i)$ have a singular flat metric with a single cone point having angle greater than or equal to $2\pi$. The disks in such cones satisfy the Euclidean isoperimetric inequality, 
$  \mbox{area}(D(\gamma_i)) \le (\ell(\gamma_i))^2/4\pi$ \cite{izmestiev}.  
As $i \to \infty$ we have that $P_i \to P$, $C(\gamma_i) \to C(\gamma)$, 
$\ell(\gamma)_i) \to\ell(\gamma) 
$ and $\mbox{area}(D(\gamma_i)) \to \mbox{area}(D(\gamma))$.  The Euclidean isoperimetric inequality holds for each smooth curve $\gamma_i$, therefore, it also holds for $D(\gamma)$ and $\gamma$.
\end{proof}

\begin{lemma} \label{planar}
Let $\gamma$ be a rectifiable curve in $\R^3$ coned to its center of mass.
The cone angle $\theta$ of $C(\gamma)$ satisfies $\theta \ge 2\pi$. If $\theta = 2\pi$ then $\gamma$ is a convex curve that lies in a flat plane in $\R^3$.
\end{lemma}
\begin{proof}
Crofton's formula for a curve on the unit 2-sphere states that the length of $\bar \gamma \subset S^2$ can be calculated by integrating, over points $p$ in the 2-sphere, the number of intersections of $\bar \gamma $ with the great circle perpendicular to $p$ and dividing by four \cite{santalo}. With the cone point $P$ at the center of mass of $\gamma$, each great circle of $S^2$ must intersect $\bar \gamma$, since otherwise $\gamma$ would lie on one side of a plane through its center of mass. 
The number of intersection points with $\bar \gamma $ is even for almost all great circles, so the length of $\bar \gamma $ is at least $8\pi/4 = 2\pi$.

Now suppose that $\gamma$ does not lie in a plane, so that $\bar \gamma$ is not contained in a great circle. Then  the convex hull of $ \gamma$ has positive volume and the center of mass of $\gamma$ lies in its interior. The length of $\bar \gamma$ is then strictly greater than $2\pi$, since $\bar \gamma$ intersects almost every great circle at least twice, and some open set of great circles at least four times. If $\gamma$ is a planar curve that is not convex, then this condition again holds. So, if $\theta = 2\pi$, then $\gamma$ is a convex planar curve.
\end{proof} 

Suppose that $p_1,p_2,p_3,p_4$ are four points in $C(\gamma)$ with $d(p_i, p_j) \ge 2$ for $i \ne j$ and $d(p_i, \gamma) \ge 2$. 
Suppose further that these four points form the vertices of a convex quadrilateral $K$ in $C(\gamma)$ and
that the cone point $P$ lies in the interior of this quadrilateral.  
Then we say that $\gamma$ has the {\em four-point property} with respect to $p_1,p_2,p_3,p_4$.
 
\begin{lemma} \label{4pts1}  
Suppose that $\gamma$ is a rectifiable curve in $\R^3$ that has the four-point property with respect to $p_1,p_2,p_3,p_4$
and $\ell(\gamma) = 8  + 4\pi$.
Then $\theta = 2\pi$, $\gamma$ lies in a plane and $K$ is a parallelogram with four edges of length two. 
\end{lemma}
\begin{proof}
We first look at the case where $\gamma$ is piecewise-smooth, so that $C(\gamma)$ has a  flat metric away from a single cone point at $P$
that has cone angle $\theta \ge 2\pi$.  
Let $D_i$  be the disk of radius two in $C(\gamma)$ centered at $p_i$, and
let $b$ be the curve forming the boundary of the convex hull of $\cup D_i$.

Closest point projection to a convex set in a CAT(0) space  from a disjoint curve decreases distances (non-strictly) \cite[Prop. 2.4] {BridsonHaefliger}
so $b$ minimizes length among all curves enclosing $\cup D_i$.
A half space in  $C(\gamma)$ whose boundary meets $b$ either meets it at a point on the boundary
of a disk $\partial D_i$ or meets it in a line segment tangent to two disks.
So $b$ contains four line segments that are 
tangent to two disks and four circular arcs, each  lying on the boundary of a disk $D_i$,
as in Figure~\ref{convexhull}.

\begin{figure}[htbp] 
    \centering
    \includegraphics[width=0.35\linewidth]{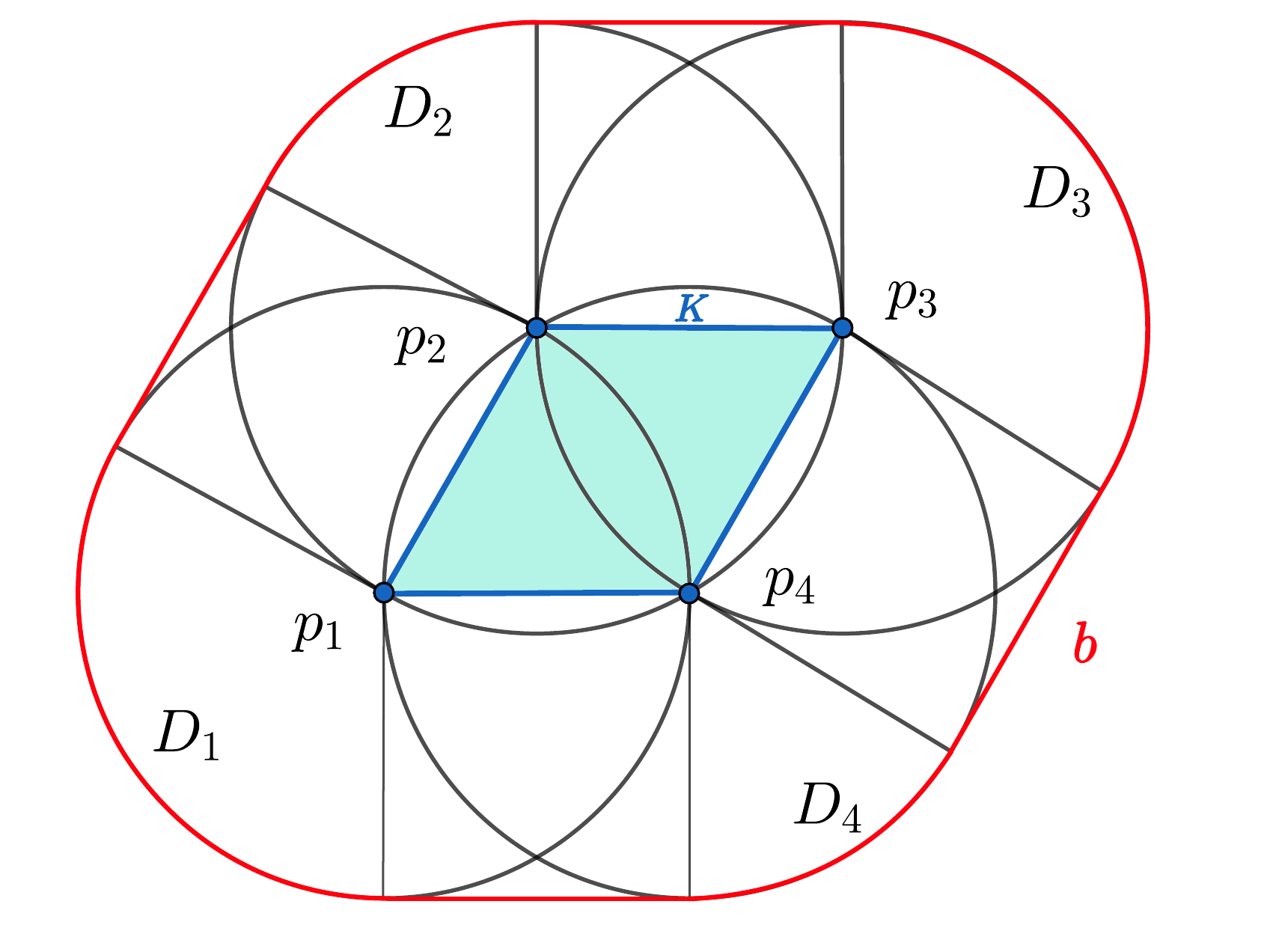}
    \caption{The convex hulls of $\cup p_i$ and $\cup D_i$ have boundaries $K$ in blue and $b$ in red, respectively.}
     \label{convexhull}
\end{figure}

Let $c$ denote the subcurve of $b$ consisting of these four circular arcs, each with constant
geodesic curvature $1/2$ and $D$ the disk enclosed by $b$.
By the Gauss-Bonnet Theorem, $2\pi \chi(D)$ equals the integral of the geodesic curvature of $c$ plus 
 the curvature contributed by the cone point.
Thus 
$$
\frac{1}{2} \ell(c) +  (2\pi - \theta)  = 2\pi 
$$
and $\ell(c) = 2 \theta $.

Each of the four line segments on $b$ is part of a
Euclidean rectangle  consisting of two radii of length $2$ and a segment of length $l_{ij} \ge  2$ on $K$ from $p_i$ to $p_j$.
The total length of the  four line segments equals $\ell(K)$,
and  $\ell(b) \ge \ell(K) +2 \theta $.

Now $\gamma$ encloses and has distance at least two from 
 $p_1, p_2, p_3, p_4$, so   $\gamma$ encloses $\cup D_i$ and  
 \begin {equation}
 \label{ell(b)}
 \ell(\gamma)  \ge  \ell(b) \ge \ell(K)  + 2 \theta.
 \end{equation}
  
But $\ell(\gamma)  = 8  + 4\pi, \ell(K) \ge 8$ and $  \theta \ge 2\pi$, so equality holds,
$\ell(b)  = \ell(\gamma)  =  8  + 4\pi$,  $\ell(K) = 8, \theta = 2\pi$ and $\gamma$ lies in a plane. 
 
We now consider the more general case where $\gamma$ is rectifiable and $\ell(\gamma) = 8  + 4\pi$.
There is a sequence of smooth curves $\gamma_j$ converging pointwise to $\gamma$ whose lengths converge to $\ell(\gamma)$ as $j \to \infty$. \
The center of mass $P_j$ of these curves converges to the center of mass $P$ of $\gamma$ and
the cones  $C(\gamma_j)$ converge to  $C(\gamma)$.
In each cone $C(\gamma_j)$ we can pick four points ${p_1}^j, {p_2}^j, {p_3}^j, {p_4}^j$ that limit to $p_1,p_2,p_3,p_4$, respectively, as $j \to \infty$. 
After passing to a subsequence, we can assume that $P_j$ lies in the interior of the convex hull $K$ of the four points ${p_i}^j$ in $C(\gamma_j)$.
Given  $\delta >0$ and sufficiently large $j$, we can assume that in addition the pairwise distance between any two points
${p_i}^j$ is larger than $2-\delta$, as is the distance from each ${p_i}^j$  to  $\gamma_j$.
The dilation $\lambda_\delta$ of $\R^3$ taking $v \in \R^3$ to $\lambda_\delta(v) = 2v/(2-\delta)$, applied to  $\gamma_j$, 
gives a sequence of somewhat longer smooth curves $\gamma'_j = \lambda_\delta(\gamma_j)$ 
that satisfy the four-point property with respect to ${p_i'}^j = \lambda_\delta({p_i}^j)$ for $j$ sufficiently large. 
We set $\delta_k = 1/k$ and for each $k$ construct a  curve  $\gamma'_k$ by applying the dilation $\lambda_{\delta_k}$ to a curve  $\gamma_j$ for
 which  $\lambda_{\delta_k}(\gamma_j)$
satisfies the four-point property with respect to $\lambda_{\delta_k}({p_i}^j)$.
Then the sequence of curves $\gamma'_k$ converge point-wise to $\gamma$ and have lengths converging to $\ell(\gamma) = 8+4\pi$.

Since the cone points $P_k$ converge to $P$, the cones $C(\gamma'_k)$ converge to $C(\gamma)$ and the smooth curves $\bar \gamma'_k \subset S^2$ converge to $\bar \gamma$, it follows that the cone angles $\theta_k$ of $\gamma'_k$ converge to the cone angle $\theta$ of $\gamma$. 
We claim that $\theta = 2\pi$.
Suppose for contradiction that $\theta > 2\pi$. Then for some $\epsilon > 0$, the cone angle $\theta_k$ of $\gamma'_k$ satisfies $\theta_k > 2\pi + \epsilon$ 
for sufficiently large $k$. Since $\ell (K) \ge 8$,
Equation~\ref{ell(b)} implies that $\ell(\gamma'_k) \ge 8 + 2\theta_k > 8+ 4\pi + 2\epsilon$ for sufficiently large $k$. 
This contradicts the assumption that $\ell(\gamma'_k) \to \ell(\gamma) = 8+ 4\pi $. We conclude that $\theta = 2\pi$ and by Lemma~\ref{planar}, $\gamma$ is contained in a plane.
\end{proof} 

Unit balls centered around two points of a thick curve are disjoint unless the distance between the points along the curve is small. 
The following explicit bound on this distance is a  consequence of Lemma~5 of \cite{CKS}.

\begin{lemma} \label{shortarc} Let $x,y$ be points on the same component of a thick link such that the open unit-radius balls $B_x$ and $B_y$ centered at $x$ and $y$ have non-empty intersection, $\mbox{int}(B_x) \cap \mbox{int}(B_y) \neq \emptyset$. Then the length of an arc of $\gamma$ between $x$ and $y$ is less than $\pi$.
Conversely, if the length of each of the two arcs of $\gamma$ between $x$ and $y$ is greater than or equal to $\pi$, then $\mbox{int}(B_x) \cap \mbox{int}(B_y) = \emptyset$. 
\end{lemma}

\section{A 2-component Gordian Unlink}

In this section, we describe a 2-component unlink $L$ and prove that it is gordian. The construction produces a family of gordian unlinks, $L(m,n)$ of which $L = L(-1,1)$ is one. The construction of $L(m,n)$ starts with the unlink shown on the left in Figure~\ref{Lmn} and proceeds by adding $m$ positive full twists between the two middle strands near the top of the link, then $n$ positive full twists between the two leftmost strands, and finally $n$ negative full twists between the two leftmost strands below the curve $\beta$. These twists cancel, so the link remains topologically unchanged. The link $L(3,-2)$ is shown in Figure~\ref{Lmn} right.  

\begin{figure}[htbp]
    \centering
    \includegraphics[width=0.3\linewidth]{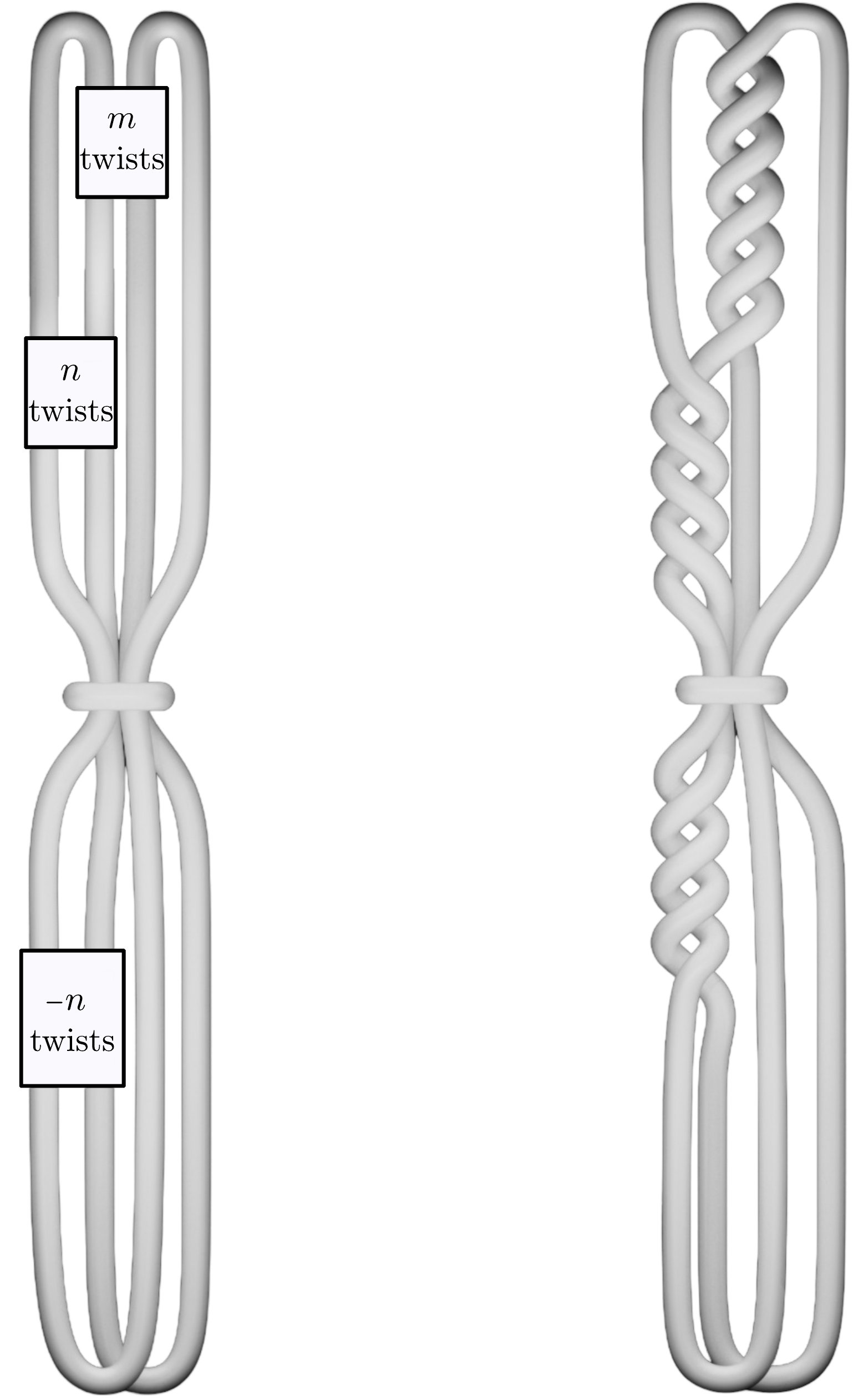}
    \caption{The family of gordian unlinks $L(m,n)$ and an example $L(3,-2)$.}
    \label{Lmn}
\end{figure}

One component $\alpha$ of $L$ intersects the $xy$-plane at four points $p_1, p_2, p_3, p_4$ that form the vertices of a square in the $xy$ plane, with $p_1$ at $(-1,-1)$, $p_2$ at $(-1,1)$, $p_3$ at $(1,1)$ and $p_4$ at $(1,-1)$. These points lie at the centers of four non-overlapping unit-radius disks. The second component $\beta$ lies in the $xy$-plane.

\begin{figure} [htbp]
    \centering
    \includegraphics[width=.6\linewidth]{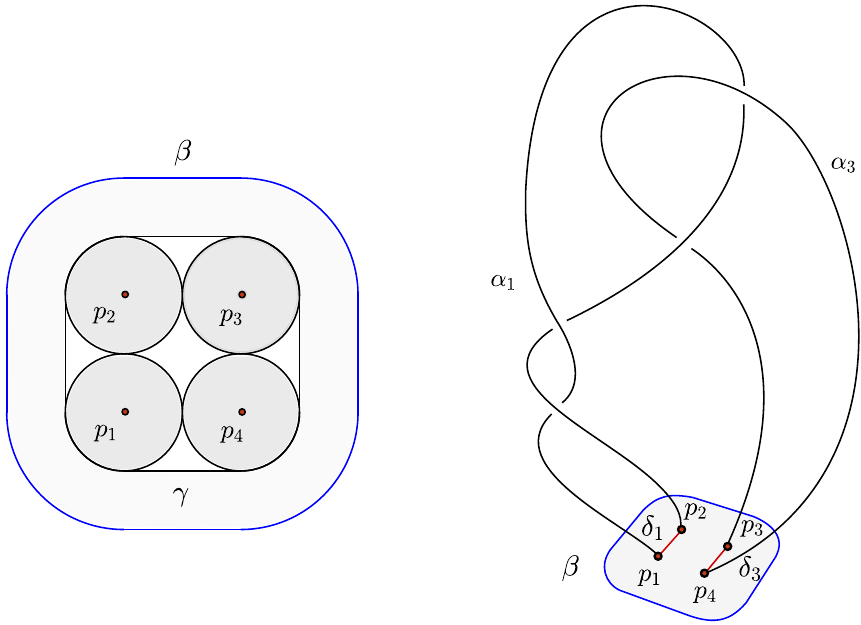}
    \caption{The planar curve $\beta$ contains four disjoint unit radius disks whose centers have distance two from $\beta$. The second component $\alpha$ passes through the centers of the disks in the indicated order. The curves $\gamma_1 = \alpha_1 \cup \delta_1$ and $\gamma_3 = \alpha_3 \cup \delta_3$ have linking number $m$,  while the curves $\gamma_2 = \alpha_2 \cup \delta_2$ and $\gamma_ 4 = \alpha_4 \cup \delta_4$ have linking number $n$.  In the figure at right, $m=-1$ and $ n=1$.}
    \label{fig:configuration}
    \end{figure}

A family of configurations of four points in the Euclidean plane is shown in Figure~\ref{fig:4discs}. In each configuration, the points form the centers of four non-overlapping unit disks, which are enclosed by a curve of length $8 + 2\pi$.  This curve in turn is enclosed by a convex curve $\beta$ whose distance from each of the four points is at least two. The length of $\beta$ in all cases is $8+4\pi$. 
%We show that these are the only planar configurations with this property.

% We will show that during a thick isotopy of $L$  $\beta$ remains planar and the four points where $\alpha$ crosses the plane containing $\beta$ always form a planar equilateral 4-gon with edge-lengths equal to two.

\begin{figure} [htbp]
    \centering
    \includegraphics[width=0.9\linewidth]{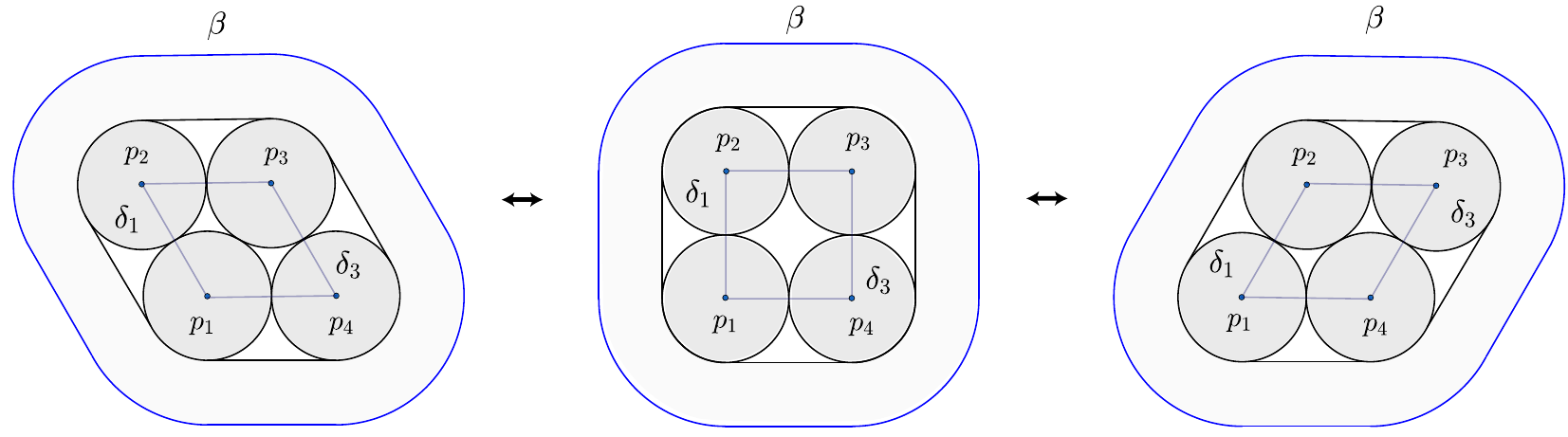}
    \caption{The link component $\beta$ is one of a family of planar convex curves, each of length $8 + 4\pi$.}
    \label{fig:4discs}
\end{figure}

The arc of $\alpha$ between $p_i$ and $p_{i+1}$ is called $\alpha_i$ (where we set $p_5 = p_1$). As can be seen in Figure~\ref{fig:configuration}, the arcs $\alpha_1$ and $\alpha_3$ are ``linked'' in the following sense. We can connect the endpoints $p_1, p_2$ of $\alpha_1$ by a straight segment $\delta_1$, and similarly connect the endpoints of $\alpha_3$ by a straight segment $\delta_3$. Then the closed curves $\gamma_1 = \alpha_1 \cup \delta_1$ and $\gamma_3 = \alpha_3 \cup \delta_3$ have non-zero linking number. Similarly, the arcs $\alpha_2$ and $\alpha_4$ are linked when closed off with the arcs $\delta_2$ from $p_2$ to $p_3$ and $\delta_4$ from $p_4$ to $p_1$

We now consider a thick isotopy $L(t)$ of $L= L(0)$, with $t \in [0, \infty]$.  Let $\alpha(t)$ and $\beta(t)$ denote the curves forming the two components of the link $L(t)$ during this thick isotopy, $C(\beta(t))$ the cone formed by coning $\beta(t)$ to the point $P(t)$ at the center of mass of $\beta(t)$,  $D(\beta(t))$ the disk bounded by $\beta(t)$ in $C(\beta(t))$, and $\theta(t)$ the corresponding cone angle. Given four points $p_1(t), p_2(t), p_3(t), p_4(t)$ in $C(\beta(t)) \cap \alpha(t)$, let $K(t)$ be the curve forming the boundary of their convex hull in $C(\beta(t))$ and $\delta_i(t)$  the geodesic in $C(\beta(t))$ from $p_i(t)$ to $p_{i+1}(t)$ (with $p_5  = p_1$).
We extend the arc $\alpha_i(t)$ to form a simple closed curve $\gamma_i(t)$ by joining its endpoints using the geodesic segment $\delta_i(t)$ on $D(\beta(t))$ that connects its endpoints $p_i(t)$ and $p_{i+1}(t)$ on $D(\beta(t))$.  Both the linking number of $(\gamma_1(0)$ and $\gamma_3(0) )$ and the linking number of $(\gamma_2(0),  \gamma_4(0) )$ are non-zero, see Figure~\ref{fig:configuration}.

We first establish a transversality property when $C(\beta(t))$ is a flat plane. Note that $\alpha(t)$ is a $C^{1,1}$ curve, with a well-defined tangent vector at each point.
%Note assuming length of beta and four point property

\begin{lemma} 
\label{transversality}
 If $\theta(t) = 2\pi$ and $D(\beta(t))$ intersect $\alpha(t) $ in four points $p_i(t)$ that satisfy the four-point property,
and $\ell(\alpha_i(t)) > \pi$, then $\alpha(t)$ is transverse to $D(\beta(t))$ at these four points.
 \end{lemma} 
\begin{proof} 
The disk $D(\beta(t))$ contains four non-overlapping radius one disks $D_i(t)$ with $D_i(t)$ centered on $p_i(t)$.
The points $p_i(t)$ form the vertices of a parallelogram with edge lengths two, interior angles between $\pi/3$ and $2\pi/3$ and shortest diagonal having length at least two.
The curve $\beta(t)$ has distance at least two from each $p_i(t)$, as shown in Figure~\ref{fig:4discs}.

We show that $\alpha(t) $  is transverse to $D(\beta(t))$ at $p_1(t)$.  
Let $v_1$ be the unit tangent vector to $\alpha(t)$ at  $p_1(t)$.
Note that if $p_1(t)$ lies on a sphere centered at a point $q$ and 
$\alpha(t)$ is disjoint from the interior of the ball bounded by this sphere
in a neighborhood of $p_1(t)$, 
then $v_1$ is tangent to this sphere.  

Since $\ell(\alpha_i(t) > \pi$, Lemma~\ref{shortarc} implies that the unit balls in $\R^3$ around any two of the points
$p_i(t)$ are disjoint. Moreover, $p_1(t)$ lies on the spheres of radius two around each of
 $ p_2(t)$ and  $p_4(t)$ since it has distance two from each.
So $v_1$ is tangent to the spheres
of radius 2 centered at $ p_2(t)$ and at $p_4(t)$.  The intersection of
the two tangent planes of these spheres at  $p_1(t)$  is the line normal to $D(\beta(t))$
at $p_1(t)$, so $\alpha(t)$ is transverse (in fact orthogonal) to $D(\beta(t))$ at $p_1(t)$.
The same argument applies at the other three points. 
 \end{proof} 

The following lemma constrains the lengths of the arcs $\{ \alpha_i(t) \}$. 

\begin{lemma} 
\label{4lengths}
Consider a thick isotopy $L(t)$ of $L$ and assume that for some $t \ge 0$, 
\begin{enumerate}
    \item $D(\beta(t)) \cap \alpha(t)$ contains four points $p_1(t), p_2(t), p_3(t), p_4(t)$ lying along $\alpha(t)$ in the indicated order.
    \item  $\ell (\alpha_i(t)) > \pi,$ $ i=1, \dots, 4$.  
    \item $K(t)$ is a convex 4-gon in $C(\beta(t))$, with vertices $p_i(t)$ and side lengths greater or equal to two.
    \item $P(t)$ lies in the interior of  $K(t) $.   
    \item   $\mbox{\rm lk}(\gamma_1(t),  \gamma_3(t) ) \ne 0 $ and  $\mbox{\rm lk}(\gamma_2(t),  \gamma_4(t)) \ne 0$.
\end{enumerate}
Then,
\begin{enumerate}
    \item  $\theta(t) = 2\pi$ and $\beta(t)$ lies in a plane.
%  \item The number of intersection points in $D(\beta(t)) \cap \alpha(t)$ is exactly four.
    \item The length of each arc $\alpha_i(t)$ satisfies $\ell (\alpha_i(t)) > 8.8$.
\end{enumerate}
\end{lemma} 
\begin{proof} 
Assumption (2) and Lemma~\ref{shortarc}  imply that $D(\beta(t))$ contains four non-overlapping radius-one disks $D_i(t)$ centered at the four points $\{p_i(t)\}$.  
Assumptions (3) and (4) imply that $\beta(t)$ has the four-point property with respect to $\{p_i(t)\}$.
Since $\ell(\beta(t))  = 8+4\pi$, Lemma~\ref{4pts1} implies that $\theta(t) =  2\pi$ and that $\beta(t)$ lies in a plane.
Moreover, the distance between any two of the $\{p_i(t)\}$ is equal to two and $K(t)$ is a parallelogram with internal angles
between $\pi/3$ and $2\pi/3$.  

We now bound from below $\{ \ell(\alpha_i(t)) \}$, utilizing the linking property of the curves $\{ \gamma_i(t) \}$.
While we assumed $ \ell(\alpha_i(t))> \pi$,  
we will show that $ \ell(\alpha_i(t))>   8.8 $. 

A point on $\alpha_1(t)$  has distance at least two from $\alpha_3(t)$ and at least $\sqrt{3}$ from $\delta_3(t)$,
and a point on $\delta_1(t)$ has distance at least $\sqrt{3}$ to $\alpha_3(t)$ and at least $\sqrt{3}$ to $\delta_3(t)$.
See Figure~\ref{fig:4balls}. So, the distance between $\gamma_1(t)$ and $\gamma_3(t)$   is at least $\sqrt{3}$.

\begin{figure}[htbp]
\centering
\includegraphics[width=0.8\linewidth]{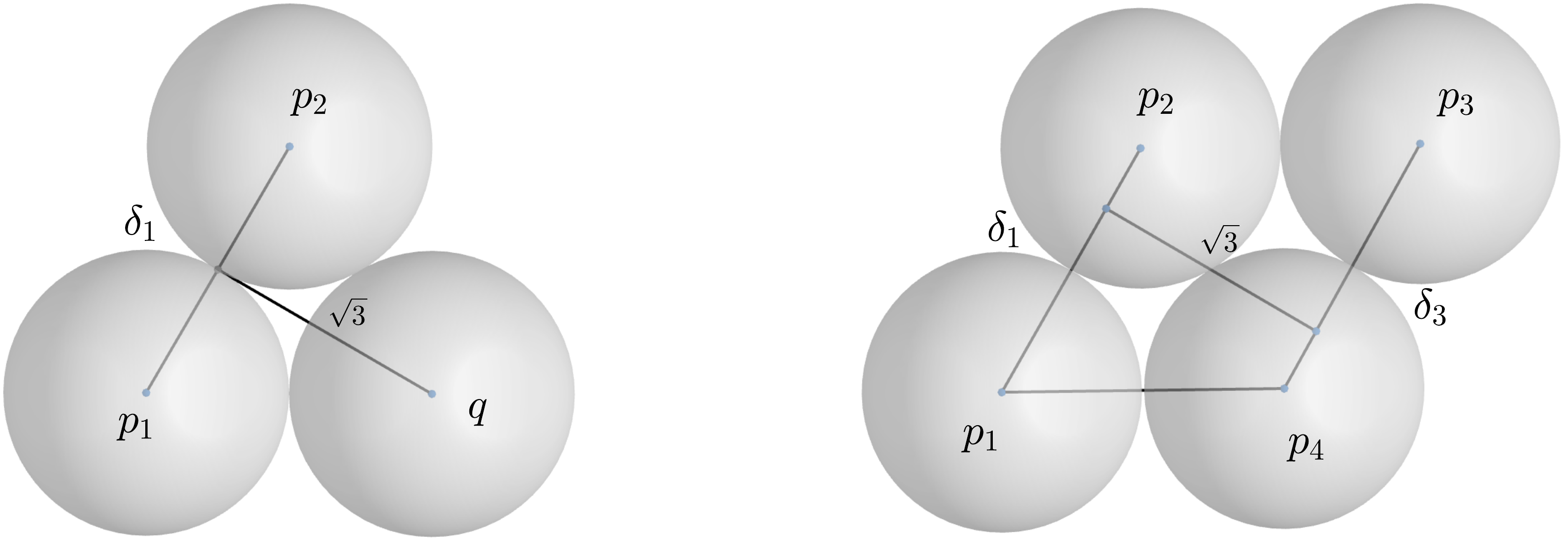}
\caption{The distance from a point $q$ in $\alpha_3$ or in $\delta_3$ to the straight segment $\delta_1 $ is at least $\sqrt{3}$.}
\label{fig:4balls}
\end{figure}

The disk $D_1$ bounded by $\gamma_1(t)$ in $C(\gamma_1(t)) $  has a cone angle of at least $2\pi$ and so is a CAT(0) space.  
$D_1$ intersects $\gamma_3(t)$ at least one point $q$,  since the linking numbers of $\gamma_1(t)$ and $\gamma_3(t) $ are assumed to be non-zero.
A disk $D_2$ of radius $\sqrt{3}$ around $q$ is contained in $D_1$ and Area$(D_2) \ge \pi (\sqrt{3})^2 = 3\pi $.
So, Area$(D_1) \ge 3\pi$, and Lemma~\ref{isoperimetric} implies that  $\ell(\gamma_1(t)) >  2\pi\sqrt{3}$. 
Subtracting $\ell(\delta_1)=2$ gives $\ell(\alpha_1(t)) \ge 2\pi\sqrt{3}-2 > 8.8$ as claimed. 
A similar argument shows that each of the arcs $\alpha_i(t), i=1,2,3,4$ have length greater than $8.8$.
\end{proof}

\noindent
{\bf Remark.} By choosing both $m$ and $n$ large, we can arrange that $\gamma_1(t)$ and $\gamma_3(t)$ have arbitrarily large linking numbers and similarly $\gamma_2(t)$ and $\gamma_4(t)$.
Then each arc $\alpha_i(t)$ is forced to be arbitrarily long, as indicated in Figure~\ref{Lmn}.

\begin{theorem} \label{gordian}
The unlink $L$ is gordian.
\end{theorem}
\begin{proof} 
If there exists a thick isotopy $L(t)$ that carries $L$  to a split link, then there is a $t>0$ 
at which $\alpha(t)$ and $\beta(t)$ lie on opposite sides of a plane in $\R^3$.  
In particular, for this, $t$ we have $\alpha(t) \cap  C(\beta(t)) = \emptyset$. 
We will show to the contrary that $\alpha(t) \cap  C(\beta(t))$ always contains four points.
Let  $W \subset [0, \infty)$ be the set of times at which the five conditions of Lemma~\ref{4lengths} hold. 
They hold at  $t=0$ by the construction of $L = L(0)$. 
We will show that $W$ is a non-empty, open and closed subset of $[0, \infty)$, so that the five conditions hold for all $t \ge 0$.  

We first show that $W$ is open.  
Suppose that $t_1 \in W$. 
Since  $\ell (\alpha_i(t_1)) > \pi$ at time $t_1$, Lemma~\ref{4lengths} implies that $\ell (\alpha_i(t_1)) > 8.8$.
The intersection  $D(\beta(t)) \cap \alpha(t)$ contains four points $ \{ p_i(t_1) \}$ and  $\alpha(t)$  intersects $D(\beta(t))$ transversely at these points by  Lemma~\ref{transversality},
so there is a positive $\epsilon$ for which  $D(\beta(t)) \cap \alpha(t)$ continues to be transverse  for $t \in [t_1, t_1+\epsilon)$. Furthermore, in this interval the pairwise distance along $ \alpha(t)$, initially greater than 8.8, remains above $\pi$, and
by Lemma~\ref{shortarc} the pairwise distance between these points in $C(\beta(t))$ is at least two.

$K(t_1)$ is a convex 4-gon in $C(\beta(t)_1)$  and this condition continues to hold  for $t \in [t_1, t_1+\epsilon)$ and $\epsilon$ sufficiently small, as the interior angles of $K(t_1)$ are at least $\pi/3$ and these  vary continuously with $t$. $P(t_1)$ lies inside and at a distance of  at least $\sqrt{3}/2$ from  $K(t_1) $, so that $P(t)$ remains in the interior of $K(t)$ for small enough $\epsilon$.
Finally, the distance between  $\gamma_1(t_1)$ and $ \gamma_3(t_1)$  is at least $\sqrt{3}$, and the curves  $\gamma_1(t)$ and $ \gamma_3(t)$ vary continuously with $t$, so they remain disjoint for small $\epsilon$.  Then $\mbox{\rm lk}(\gamma_1(t),  \gamma_3(t))$ remains unchanged and non-zero for  $ t \in [t_1, t_1+\epsilon)$, as does  $\mbox{\rm lk}(\gamma_2(t),  \gamma_4(t))$.
Thus, all assumptions of  Lemma~\ref{4lengths} hold for  $ t \in [t_1, t_1+\epsilon)$.

We now show that $W$ is closed. Assume that $(t_0, t_1) \subset W$. 
 Then $\theta(t) = 2\pi$ and $\beta(t)$ is contained in a plane for each $t \in (t_0, t_1) $ and therefore also at $t=t_1$.
As $t \to t_1$ the points $\{ p_i(t) \}$ converge to $p_i(t_1)$ and have distance along $\alpha(t)$ that is greater than 8.8. Thus, the distance along $\alpha(t_1)$ between two of the points 
$\{p_i(t_1) \}$ is greater or equal to 8.8, and by Lemma~\ref{shortarc} their distance in $C(\beta(t_1))$ is at least two. Moreover, the interior angles of $K(t)$ are at least $\pi/3$ for $t \in (t_0, t_1) $ and these  vary continuously with $t$ so $K(t_1)$ remains a convex quadrilateral. $P(t)$ lies inside and at a distance of  at least $\sqrt{3}/2$ from  $K(t) $, so that $P(t_1)$ remains in the interior of $K(t_1)$. Thus $\gamma_{t_1}$ satisfies the four-point property with respect to  $\{ p_i(t) \}$.
Finally, the distance between  $\gamma_1(t)$ and $ \gamma_3(t)$  is at least $\sqrt{3}$ for $t \in (t_0, t_1) $, as is the distance between their
limiting curves $\gamma_1(t_1)$ and $ \gamma_3(t_1)$.  So the linking numbers remain constant on $ (t_0, t_1]$
and $\mbox{\rm lk}(\gamma_1(t_1),  \gamma_3(t_1) ) \ne 0 $ and  $\mbox{\rm lk}(\gamma_2(t_1),  \gamma_4(t_1)) \ne 0$.
Each of the five conditions in  in Lemma~\ref{4lengths} 
hold for $L(t_1)$ and $W$ is closed.

Since $W$ is open and closed, the five conditions in Lemma~\ref{4lengths} 
hold for $t \in (t_0, \infty) $. In particular, 
$D(\beta(t)) \cap \alpha(t)$ contains four points for all $t \ge 0$. If the two components of the link could be separated by a plane  at some finite time $t$ then we would have $D(\beta(t)) \cap \alpha(t) = \emptyset$.  We conclude that the link cannot be split by a thick isotopy, and thus that $L$ is a gordian unlink.
\end{proof}

An identical argument establishes the following result.
\begin{corollary}The unlinks $L(m,n)$, $n \ne 0,m \ne 0$, are gordian.
\end{corollary} \label{lkgordian}

\noindent
{\bf Final Remarks.}
\begin{enumerate}
\item The construction above can be used to generate $n$-component gordian unlinks for any $n>1$. To do this, we stack $(n-1)$ parallel copies of $\beta$, each having length $8+4\pi$.

\item An {\em ideal link} is a thick link which cannot be shortened within its thick isotopy class. By minimizing the length of the component $\alpha$ in the thick isotopy class of the link $L$, we can find an ideal configuration of the 2-component unlink which is not standard. We know from Cantarella, Kusner and Sullivan \cite{CKS} that this configuration is $C^{1,1}$, but we do not have an explicit description.

\item  With the length of $\beta$ equal to $8+ 4\pi$ the component $\beta$ remains in a plane during a thick isotopy, though
its shape may change among the quadrilaterals shown in Figure~\ref{fig:4discs}. The set of trivial 2-component thick links is closed, so a neighborhood of $L$ is not split. Thus, the link $L$ remains gordian if the length of $\beta$ is allowed to be somewhat larger, but $\beta$ is no longer constrained to be planar. We have not computed the maximal length for
$\beta$ in which $L$ remains gordian.
\end{enumerate}

\end{document}